\documentclass{article}
\usepackage{csquotes}
\usepackage[utf8]{inputenc}
\usepackage[backend=biber,style=alphabetic]{biblatex}
\usepackage[english]{babel}
\usepackage[margin=1in]{geometry}
\usepackage{amsmath,amssymb,amsfonts,amsthm,mathtools}
\usepackage{braket} 
\usepackage{xcolor}
\usepackage{tikz}
\usetikzlibrary{fit}
\usepackage{quantikz}
\usepackage[inline]{enumitem}

\usepackage{hyperref}
\usepackage[nameinlink]{cleveref}

\hypersetup{colorlinks=true,linkcolor=blue,filecolor=magenta,urlcolor=cyan}
\newtheorem{theorem}{Theorem}[section]

\newtheorem{lemma}[theorem]{Lemma}
\theoremstyle{remark}

\theoremstyle{definition}
\newtheorem{definition}{Definition}[section]
\newtheorem{prop}{Proposition}

\newcommand{\R}{\mathbb{R}}
\newcommand{\C}{\mathbb{C}}

\newcommand{\vdx}{\vert dx\vert}
\title{Delta functions, a local coarea formula and marginal distributions}
\author{Ranee Brylinski\thanks{She passed away before this paper was submitted} and Jean-Luc Brylinski}
\date{October 2026}
\begin{document}
\maketitle
\begin{abstract}
 We define multidimensional delta functions $\delta(f)$ as generalized functions by pulling back the Dirac delta function $\delta(x)$, as well as products of such. We prove various formulas including a local version "on shell" of the coarea formula and we use delta functions to compute marginal probability distributions on spheres.

\end{abstract}

\section{Introduction}\label{sec:intro}

The purpose of this paper is to facilitate for mathematicians the use  of the multi-dimensional $\delta$ functions
in computing volumes and multiple integrals. As an example, here is a formula on page 67 of a paper of Sommers and Zyckowski \cite{S-Z}

\begin{equation}
\int D\rho\frac{\delta(Tr \rho-1)}{det\rho^{1/2}}\theta(\rho)\, \prod_{\nu <\mu}
\frac{2}{\rho_{\nu}+\rho_{\mu}}
\end{equation}
    an integration carried over the set of positive hermitian matrices, the term $\delta(Tr \rho-1)$
having the effect  of localizing the integrand to the matrices of trace $1$, i.e. the set of mixed quatum states.
\vskip .12 in

We will give here a direct construction of a generalized function $\delta(f)$ for any smooth function $f$ on a smooth manifold $M$ of dimension $n$, whose differential vanishes nowhere on the zero locus $Y$ of $f$.

\begin{prop}\label{def}
There is a canonical construction of a generalized function $\delta(f)$ such that for any $(n-1)$-form form $\nu$ on a neighborhood of $Y$ in $M$ and any smooth function $g$ we have
\begin{equation}
\int_M\delta(f)g\,df\wedge\nu=\int_Y g\nu
\end{equation}
\end{prop}

An equivalent formulation, maybe more palatable,
is that for any volume form $\omega$ on $M$, we have

\begin{equation}
\int_M\, \delta(f)\vert\omega\vert=\int_Y\, \bigl[\vert\frac{\omega}{df}\vert\bigr]_{/Y}
\end{equation}

where $\large\frac{\omega}{df} $ is the Gelfand-Leray
$(n-1)$-form such that

\begin{equation}
df\wedge \frac{\omega}{df}=\omega
\end{equation}
\vspace{.112in}

The $(n-1)$-form $\frac{\omega}{df}$ is not unique, but its restriction to $Y=f^{-1}(0)$ is unique, see for instance \cite{A-GZ-V}.

\begin{prop}\label{local}
 Given local coordinates $(x_1,\cdots,x_n)$ on an open subset $U$ of $M$ such that $\frac{\partial f}{\partial x_1}\neq 0$ on $U$, we have

 \begin{equation}
 \int_M \delta(f)\, |dx_1\cdots dx_n|
 = \int_Y \left|\left(\frac{\partial f}{\partial x_1}\right)\right|^{-1}\, |dx_2\cdots dx_n|\Big\rvert_Y.
 \end{equation}
\end{prop}

\section{$\delta(x)$ as a generalized function}

Next year will be the hundred year anniversary of Dirac's function $\delta(x)$ on the real line $\R$, As a \underline{generalized function} it is characterized by its integration against a smooth compactly-supported $1$-form $f|dx|$

\begin{equation}
\int_{\R} \delta(x) f(x) \vdx= f(0)
\end{equation}

This was exactly the way Dirac defined it \cite{dirac}, as  a "function" that you cannot evaluate  at points but that you can integrate against a function $f(x)$.

Laurent Schwartz \cite{schw} viewed the delta function in the same way, with the addition that the functional on $f(x)$ is continuous in a suitable sense, that we recall later.

We take here the approach from \cite{G-S} that $\delta(x)$ is a generalized function, that is to say, a functional on smooth compactly-supported densities $d\mu = f(x)\vdx$.

\begin{equation}
\mu\mapsto\int_{\R}\delta(x)\, d\mu(x)
\end{equation}

The quantity $\delta(x)\, \vert dx\vert$ is itself a probability measure supported at $0$ (the unique one), called the Dirac measure.

The scaling invariance of the Dirac measure yields

\begin{equation}
\delta(ax)=|a|^{-1}\delta(x)
\end{equation}

We note the important and very consequential property of $\delta(x)$, often called the \underline{total area property}.

\begin{prop}
We have the equality between generalized functions over $\R$
\begin{equation}
\int_{\R}\, \delta(x-a)\vert da\vert =1
    \end{equation}
  Here $\delta(x-a)$ is the generalized function which  evaluates to $f(a)$ when evaluated  against a smooth compactly density $f(x)\vert dx\vert$.  Equivalently, translations on $\R$ act naturally on smooth compactly supported  densities, hence dually on generalized functions.
  
  The integral above is understood as a weak integral (aka Pettis integral).
\end{prop}

\begin{proof}
Since the integral of generalized funcions  in question is a weak (aka Pettis) integral, we compute it against a test smooth density $g(x)\,|dx|$:
\begin{equation}
\int_{\R}\!\left(\int_{\R}\delta(x-a)\,|da|\right) g(x)\,|dx|
=\int_{\R}\!|da|\int_{\R}\delta(x-a)\,g(x)\,|dx|
=\int_{\R} g(a)\,|da|.
\end{equation}
Since this equals $\int_{\R} g(x)\,|dx|$ for all $g$, we have $\int_{\R}\delta(x-a)\,|da|=1$ as generalized functions.
\end{proof}

\section{Construction of $\delta(f)$}

The method is very simple: define $\delta(f)=f^*\delta(x)$ using pullback of generalized functions. To make sense of this, we first need to define generalized functions.

\begin{definition}
A  generalized function $g$ on a smooth manifold $M$ is a continuous functional on the space $Dens_c(M)$ of smooth compactly supported densities on $M$., where $Dens_c(M)$ is vieweed a the direct limit of the Frechet spaces $Dens_K(M)$ of smooth densities supportee on compact subsets $K$ of $M$. The space of generalized functions on $M$ is denoted by $\mathcal D^{\prime}(M)$

\end{definition}

By definition of the direct limit topology, a linear functional on $Dens_c(M)$ is continuous  if its restriction to each $Dens_K(M)$ is continuous.

\begin{prop}
If the function $f$  has nowhere vanishing differentiaL  $df$, there is a well-define fiber integration map $\int_f::Dens_c(M)\to Dens_c(\R)$.
Then the pull-back map $f^*: \mathcal D^{\prime}(\R)\to \mathcal D^{\prime}(M)$ is well-defined and continuous.

    \end{prop}

    \begin{prop}
    
    Under the assumption that $df$ vanishes nowhere on the locus $Y=f^{-1}(0)$, we can define $\delta(f)=f^*\delta(x)$.

        \end{prop}

        To prove Proposition \ref{def}, we need to compute
        $\int_M\delta(f)g\vert df\wedge \, d\nu\vert$.

        \begin{align*}
\int_M\delta(f)g\vert df\wedge  d\nu\vert&=\int_M f^*\delta(x)
g\vert df\wedge d\nu\vert\\
&=\int_{\R}\delta(x)\int_f\, g\, \vert df\wedge d\nu\vert\\
&=\int_{\R}\delta(x)dx\int_fg
\vert d\nu\vert\\
&=\int_Yg\, \vert d\nu\vert
\end{align*}
 where we have moved the $df$ inside the fiber integration  to a $dx$ on the real line.
     Now to prove Prosition \ref{local}, we just need to observe that $\bigl(\frac{\partial f}{\partial x_1}\bigr)^{-1}df\wedge dx_2\cdots \wedge dx_n=dx_1\wedge\cdots\wedge dx_n$, so that $\bigl(\frac{\partial f}{\partial x_1}\bigr)^{-1}dx_2\cdots \wedge dx_n$ is an allowable choice for the $(n-1)$-form $\nu$.

     \vskip .12 in
     Now we get a nice geometric formula in the riemannian situation, where $M$ has a riemannian metric $ds^2$ with volume density $|\omega|$. Let $|\nu|$ be the volume form on $Y=f^{-1}(0)$ for the induced metric $(ds^2)_Y$.

     \begin{prop}
We have $\int_M\delta(f)g\omega =\int_Y\, \vert\vert \nabla f\vert\vert^{-1} g\nu$

         \end{prop}
         \begin{proof}
Working on a local tubular neighborhood $T$ of $Y$ in $M$, we have local coordinates $(x_1,\cdots,x_{n-1},r)$ where $r$ is the radial coordinate perpendicular to $Y$. Then on $Y$ itself (but not on a neighborhood of $Y$) we have $\vert\omega\vert=\vert \nu\cdot dr\vert$. Also on $Y$ we have $df=d(f_{/Y})+\frac{\partial f}{\partial r}dr$, so that $\bigl(\frac{\omega}{df}\bigr)_{Y}=(\frac{\partial f}{\partial r})^{-1}\nu=\lVert \nabla f\rVert^{-1} \nu$.
         \end{proof}

  We now have the folliowing total area property, just as in the case  of $\delta(x)$. 

  \begin{prop}\label{tot2}
If the mapping $f:M\to\R$ has non-zero differential at each point of $M$, we have the equation
\begin{equation}
\int_{\R}\, \delta(f-a) da =1
\end{equation}
\end{prop}

\begin{proof}
Apply the pullback map  $f^*$ to the equation $\int_{\R}\, \delta(x-a)da=1$, noting that the pull-back map respects weak integrals.
    \end{proof}

\section{Product of delta functions}\label{sec:product}

We cannot expect the product $\delta(f_1)\cdots \delta(f_k)$ of several sdelta functions on a manifold $M$ to be alwys defined. For instance, the square of $\delta(x)$ on $\R$ makes no sense. However
the product $\delta(x_1)\cdots\delta(x_k)$ is well-defined as a generalized function on $\R^n$ by the formula
\begin{equation}
 \int_{\R^n}\, \delta(x_1)\cdots\delta(x_k)\,g\,\lvert dx_1\cdots dx_n\rvert
 =\int_Y\,g\, \lvert dx_{k+1}\cdots dx_n\rvert
\end{equation}

where $Y$ is the coordinate subspace of $\R^n$ defined by the equations $x_1=\cdots=x_k=0$
\vskip .12 in
 \begin{prop}

Let $(f_1,\cdots,f_k)$ be smooth functions on a smooth manifold $M$ such that the smooth mapping $F=(f_1,\cdots,f_k):M\to\R^k$ is a submersion at all points of $Y=F^{-1}(0)$, then the pullback 
$F^*(\delta(x_1)\cdots\delta(x_k))$  is well-defined, and will be denoted $\delta(f_1)\cdots\delta(f_k)$. This product of delta functions is commutative and associative under the submersive assumptions. 
     \end{prop}
 
\begin{proof}

  When $F$ is a submersion, the fiber integration map
  $\int_F::Dens_c(M)\to Dens_c(\R^k)$ is well-defined and is continuous for the inductive limit Fr\'echet topologies. The dual map $F^*$ is a continuous map from generalized functions on $\R^k$ to generalized functions on $M$.
\end{proof}

We then have the formula
\begin{equation}
\int_M\delta(f_1)\cdots\delta(f_k) g\vert\omega\vert=\int_Y
g\vert \frac{\omega}{df_1\cdots df_k}\vert
\end{equation}
where the $(n-k)$-differential form $\frac{\omega}{df_1\cdots df_k}$ on $M$ is defined by the condition
\begin{equation}
df_1\wedge\cdots\wedge df_k\wedge \frac{\omega}{df_1\cdots df_k}=\omega
\end{equation}
The $(n-k)$-form $\frac{\omega}{df_1\cdots df_k}$ is not unique, but its restriction to $Y$ is unique, just as in the $k=1$ case.

The construction of the product $\delta(f_1)\cdots\delta(f_k) $ has  a remarkable flexibility with respect to a change of  variables in the following sense. We will say that $(g_1,\cdots,g_k)$ are a new set of variables defining $Y$
 if the corersponding mapping $G:M\to\R^k$ is of the form $G=\Phi\circ F$ where $\Phi:\R^k\to\R^k$ is a smooth mapping which maps $0$ to $0$ which has a smooth inverse in a neighborhood of $0$.
\begin{prop}
If $(g_1,\cdots,g_k)$ are a new set of variables defining $Y$, then we have
\begin{equation}
\delta(g_1)\cdots\delta(g_k)=\vert J\vert^{-1}\, \delta(f_1)\cdots\delta(f_k)
\end{equation}
where $J=\det\left(\frac{\partial g_i}{\partial f_j}\right)$.
\end{prop}

\begin{proof}
We have the equation $dg_1\wedge\cdots\wedge dg_k=J\, df_1\wedge\cdots\wedge df_k$ which yields immediately
\begin{equation}
 \frac{\omega}{dg_1\cdots dg_k}=J^{-1}\frac{\omega}{df_1\cdots df_k}.
\end{equation}
\end{proof}

In the case where $M$ has a riemannian metric with volume density $\vert \omega\vert$ and $\vert\nu\vert$ is the volume density on $Y$ given by the induced metric, we have the following local version of the smooth coarea formula:

\begin{prop}\label{local}

\begin{equation}
\int_M\, \delta(f_1)\cdots \delta(f_k)\, g\,|\omega|
=\int_Y\, |NJ(F)|^{-1}\, g\, \vert\nu\vert.
\end{equation}
where $NJ(F)$ is the \underline{normal jacobian} of $F$, i.e., the determinant of the restriction of the differential of $F$ to the normal bundle of $Y$. See e.g. \cite{Ch} for the normal jacobian.
\end{prop}

\begin{proof}
Fix $y\in Y$ and consideer  the orthogonal decomposition $T_yM=T_yY\oplus N_y$ with $\dim N_y=k$. Let $\xi_1,\dots,\xi_k$ be an orthornormal basis of $N_y^*$ such that
$df_1\wedge\cdots\wedge df_k = NJ(F)(y)\,\xi_1\wedge\cdots\wedge\xi_k$ on $N_y$. Contracting the volume density $\vert\omega\vert$ by $|df_1,\dots,df_k|$ therefore produces the induced riemannian density on $T_yY$ scaled by $NJ(F)(y)^{-1}$, i.e. $NJ(F)^{-1}\vert\nu\vert$.
\end{proof}

We again have the total area property

\begin{prop}\label{total}
If the mapping $F:M\to \R^k$ is a submersion, we have
\begin{equation}
\int_{\R^k}\, \delta(f_1-a_1)\cdots \delta(f_k-a_k)\,\vert da_1\cdots da_k\vert=1.
\end{equation}
\end{prop}

Thos gives the following 
\begin{theorem}
We have the following smooth coarea formula
\begin{equation}
\mathop{\int}_{\R^k}\, g\omega=\int_{\R^k}\, 
    |da_1\cdots da_k|\bigl(\mathop{\int}_{F^{-1}(\bar{a} )}|NJ(F)|^{-1}g
\nu_{\bar{a}}\bigr),
\end{equation}
where $\bar a=(a_1,\cdots,a_k)$ and $\nu_{\bar a}$ is the riemannian volume form on $F^{-1}(\bar a)$.
\end{theorem}

\begin{proof}
        Combine the local coarea formula \ref{local} with the total area property \ref{total}
\end{proof}

\section{Computation  of marginal distributions}\label{marginal}

We will consider probaility distributions on  a smooth compact manifold $M$ which are given by a smooth density $|\omega|$ over $M$. Given smooth functions
$(f_1,\cdots,f_k)$ over $M$, viewed as random variables over $M$, we can form the marginal distribution of these functions, which is the push-forward $F_*|\omega|$ of $|\omega|$ under the mapping $F:M\to \R^k$.

\begin{prop} Assume that the mapping $F:M\to\R^k$ is submersion. Then the marginal probability distribution of the $f_i$'s has a smooth pdf $g(\bar a)$ (probability density function) given by $g(\bar a)=\int_M, \delta(f_1-a_1)\cdots\delta(f_k-a_k)|\omega|$
\end{prop}

\begin{proof}
We have
\begin{align*}
\int_M\, \delta(f_1-a_1)\cdots\delta(f_k-a_k)|\omega&|=\int_M\, F^*\bigl(\delta)(x_1-a_1)\cdots\delta(x_k-a_k)|\omega|\\
&= \int_M\, \delta(x_1-a_1)\cdots \delta(x_k-a_k)F_*|\omega|\\
&=g(\bar a)
    \end{align*}

    \end{proof}

Let $S^{n-1}$ be the standard sphere of radius $1$ in $\R^n$, equipped with its unique rotation-invariant volume density $|\omega|$. We will compute the marginal distribution for the $k$ coordinate functions $(x_1,\cdots,x_k)$.

\begin{lemma}
The mapping $F:S^{n-1}\to \R^k$ given by $F(x)=(x_1,\cdots,x_k)$ is a submersion.
\end{lemma}

\begin{proof}
Consider the mapping $h:\R^n\to \R^{k+1}$ given by
\[h(x)=(x_1,\dots,x_k,\,x_1^2+\cdots+x_n^2).\]
It is a submersion since the differentials of these $k+1$ functions are linearly independent everywhere. Restricting $h$ to the level set $x_1^2+\cdots+x_n^2=1$ shows that $F$ has rank $k$ on $S^{n-1}$.
\end{proof}

Then we can form the product $\delta(f_1-a_1)\cdots\delta(f_k-a_k)$ as a generalized function over $S^{n-1}$.

\begin{lemma}
The invariant volume density $|\omega|$ on $S^{n-1}$ satisfies
\begin{equation}
|\omega|=2\,\frac{|dx|}{d(r^2)},\qquad r^2=\sum_{i=1}^n x_i^2.
\end{equation}
\end{lemma}

\begin{proof}
In polar coordinates, $|dx|=r^{n-1}|\omega|\,|dr|$, hence $|dx|/|dr|=r^{n-1}|\omega|$. On the locus $r=1$ we have $d(r^2)=2r\,dr=2\,dr$, giving the stated identity.
\end{proof}

We are now ready to compute the pdf $g(\bar a)$ of the random variables $(x_1,\cdots,x_k)$ on $S^{n-1}$. We first note that changing the function $r^2$ to $x_{k+1}^2+\cdots+x_n^2$ in the $k+1$ functions $(x_1,x_2,\cdots,x_k,x_1^2+\cdots+x_n^2)$ introduces no jacobian factor.

\begin{align*}
g(\bar a)
&=\int_{S^{n-1}} \delta(x_1-a_1)\cdots\delta(x_k-a_k)\,|\omega|\\
&=\frac{1}{2}\int_{\R^n} \delta(x_1-a_1)\cdots\delta(x_k-a_k)\,\delta\bigl(x_{k+1}^2+\cdots+x_n^2-(1-|\bar a|^2)\bigr)\,|dx|\\
&=\frac{1}{2}\int_{\R^{n-k}} \delta\bigl(x_{k+1}^2+\cdots+x_n^2-(1-|\bar a|^2)\bigr)\,|dx_{k+1}\cdots dx_n|\\
&=V_{n-2}(1-|\bar a|^2)^{\frac{n-2}{2}}
\end{align*}
where $V_{n-2}$ is the volume of the radius one solid sphere of dimension $n-2$.

In order to see a valuable limit for $n\to\infty$, it is  appropriate to rescale the random variables $x_i$ of variance $1/n$ into the variables $y_i=n^{1/2} x_i$ which have variance $1$. Then up to a constant factor we have the pdf
$p(\bar y)\propto(1-\frac{y^2}{n})^{\frac{n-2}{2}} $. By Euler's formula this tends to $\propto \exp(-y^2)$. This fits with the well-known result that the marginal distribution of $k$ coordinate functions on $S^{n-1}$ tends to a Gaussian distribution as $n\to\infty$  and $k=o(n)$, see \cite{di-fr}.

In quantum computation, the random variable $p=x_1^2+x_2^2$ is of great interest as it gives the probability of a measurement result when $n=2m$, so that $\R^n=\C^m$ and $x_1^2+x_2^2=|z_1|^2$. In the $n\to\infty$ limit, this random variable has the pdf proportional to $exp(-p)$, so is an exponential distribution, aka as Porter-Thomas distribution. See \cite{bo} and \cite{google} for  thorough discussion,s from both the theoretical and experimental standpoints, in the context of the Google quantum supremacy experimants.

\printbibliography

@article{S-Z,
   title={Bures volume of the set of mixed quantum states},
   volume={36},
   ISSN={1361-6447},
   url={http://dx.doi.org/10.1088/0305-4470/36/39/308},
   DOI={10.1088/0305-4470/36/39/308},
   number={39},
   journal={Journal of Physics A: Mathematical and General},
   publisher={IOP Publishing},
   author={Sommers, Hans-Jurgen and Zyczkowski, Karol},
   year={2003},
   month=Sept, pages={10083–10100} }

@book{A-GZ-V,
title={Singularities of Differentiable Maps. Volume2},
publisher={Birkhauser},
author={V.I. Arnold and S.M. Gusein-Zade and A.N. Varchenko},
year={1998}  }

@book{Ch,
title={Riemannian Geometry. A Modern Introduction. Second Edition},
publisher={Cambridge Univwersity Press},
author={I.Chavel},
year={2006}  }

@book{dirac,
title={The principles of Quantum Mechanics},
publisher={Orford Univwersity Press},
author={ P. A. M. Dirac },
year={1930  }}

@book{schw,
title={Theorie des Distributions vol, 1 and 2},
publisher={Hermann, Paris},
author={ L.Schwratz},
year={1950-1951  }}

@article{di-fr,
title= {A dozen de Finetti‐style results: In search of a theory},
journal={ Ann. Inst. H. Poincaré},
author={P. Diaconis and D. Freedman},
volume={23},
pages={397-423},
year={1987}}

@article{google,
title= {Quantum supremacy using a programmable superconducting processor},
journal={ Nature},
author={F.Arute et al.},
volume={574},
pages={505-510},
year={2019}}

@article{bo,
title= {Characterizing quantum supremacy in near-term devices},
journal={ Nature Physics},
author={S.Boxio},
volume={14},
pages={905-600},
year={2018}}

@book{G-S,
title={Generalized Functions, Volume 2: Spaces of Fundamental and Generalized Functions},
publisher={AMS Chelsea Publishing},
author={I. M. Gel'fand and G. E. Shilov},
year={1968}}
\end{document}